\documentclass[12pt]{amsart}
\usepackage{amssymb}
\usepackage{stmaryrd}
\usepackage{fancyhdr}
\usepackage{amscd,latexsym,amsthm,amsfonts,amssymb,amsmath,amsxtra}
\usepackage[mathscr]{eucal}
\usepackage{geometry}
\usepackage[colorlinks=true,urlcolor=blue,bookmarks=true,bookmarksopen=true,citecolor=blue]{hyperref}
\usepackage[all]{xy}
\renewcommand\theequation{\thesection.\arabic{equation}}

\newcommand{\BA}{{\mathbb {A}}}

\newcommand{\BC}{{\mathbb {C}}}

\newcommand{\BQ}{{\mathbb {Q}}}

\renewcommand{\CD}{{\mathcal {D}}}

\newcommand{\CF}{{\mathcal {F}}}

\newcommand{\CO}{{\mathcal {O}}}

\newcommand{\RG}{{\mathrm {G}}}

\newcommand{\GL}{{\mathrm{GL}}}

\newcommand{\Hom}{{\mathrm{Hom}}}

\newcommand{\Ind}{{\mathrm{Ind}}}
\newcommand{\ind}{{\mathrm{ind}}}

\renewcommand{\Re}{{\mathrm{Re}}}

\newcommand{\bs}{\backslash}

\def\diag{{\rm diag}}

\newtheorem{thm}{Theorem}[section]
\newtheorem{cor}[thm]{Corollary}

\newtheorem{prop}[thm]{Proposition}

\newtheorem {assu}[thm]{Assumption}
\newtheorem {ques/conj}[thm]{Question/Conjecture}

\newtheorem{rmk}[thm]{Remark}

\newtheorem{exmp}[thm]{Example}

\makeatletter

\newcommand{\Rmnum}[1]{\expandafter\@slowromancap\romannumeral #1@}
\makeatother

\begin{document}
\renewcommand{\theequation}{\arabic{equation}}
\numberwithin{equation}{section}

\title[Top Fourier coefficients of automorphic forms]{On top Fourier coefficients of certain automorphic representations of $\GL_n$}

\author{Baiying Liu}
\address{Department of Mathematics\\
Purdue University\\
150 N. University St\\
West Lafayette, IN 47907, USA}
\email{liu2053@purdue.edu}

\author{Bin Xu}
\address{School of Mathematics\\
Sichuan University\\
No. 29 Wangjiang Road\\ 
Chengdu, 610064, People's Republic of China}
\email{binxu@scu.edu.cn}




\thanks{The research of the first named author is partially supported by NSF grant DMS-1702218 and by start-up funds from the Department of Mathematics at Purdue University. The second named author is partially supported by NSFC grant No.11501382 and by the Fundamental Research Funds for the Central Universities}

\begin{abstract}
In this paper, we study top Fourier coefficients of certain automorphic representations of $\GL_n(\BA)$. 
In particular, we prove a conjecture of Jiang on top Fourier coefficients of isobaric automorphic representations of $\GL_n(\BA)$ of form
$$
\Delta(\tau_1, b_1) \boxplus \Delta(\tau_2, b_2) \boxplus \cdots \boxplus \Delta(\tau_r, b_r)\,,
$$
where $\Delta(\tau_i,b_i)$'s are Speh representations 
in the discrete spectrum of $\GL_{a_ib_i}(\BA)$ with $\tau_i$'s being unitary cuspidal representations of $\GL_{a_i}(\BA)$, and $n = \sum_{i=1}^r a_ib_i$.
Endoscopic lifting images of discrete spectrum of classical groups form a special class of such representations. The result of this paper will facilitate the study of automorphic forms of classical groups occurring in the discrete spectrum. 
\end{abstract}

\maketitle



\section{Introduction}\label{section: intro}

Fourier coefficients are important in the study automorphic forms. For example, Whittaker-Fourier coefficients play an essential role in the theory of constructing automorphic $L$-functions, either by Rankin-Selberg method or by Langlands-Shahidi method. In general, there is a framework of attaching Fourier coefficients to nilpotent orbits (see \cite{GRS03, G06, J14, GGS17a}, and also \S \ref{def of FC's} for details), which has also been used in theory of automorphic descent (see \cite{GRS11}). 
Let $F$ be a number field and $\BA$ be its ring of adeles. Let $\mathrm{G}$ be a connected reductive group defined over $F$. 
One important topic in the theory of Fourier coefficients is to study all nilpotent orbits providing nonzero Fourier coefficients for a given automorphic representation $\pi$ of $\RG(\BA)$. We denote the set of all such nilpotent orbits by $\mathfrak{n}(\pi)$. 
The subset of maximal nilpotent orbits ${\mathfrak{n}}^m(\pi)$ in $\mathfrak{n}(\pi)$ under the natural ordering of partitions is particularly interesting.
For classical groups, nilpotent orbits are parameterized by partitions of certain integers (see \cite{CM93, W01}), and in such cases, a relatively easier question is to characterize the sets of partitions $\mathfrak{p}(\pi)$ and $\mathfrak{p}^m(\pi)$ parameterizing nilpotent orbits in 
$\mathfrak{n}(\pi)$ and $\mathfrak{n}^m(\pi)$, respectively. A folklore conjecture is that all nilpotent orbits in $\mathfrak{n}^m(\pi)$ belong to the same geometric orbit (namely over the algebraic closure $\overline{F}$), this means that the set $\mathfrak{p}^m(\pi)$ is a singleton in the cases of classical groups. 
The properties of $\mathfrak{n}(\pi)$, $\mathfrak{n}^m(\pi)$, $\mathfrak{p}(\pi)$, and $\mathfrak{p}^m(\pi)$ have been studied extensively in many papers, for example, 
\cite{GRS03, G06, J14, JL13, JL15, JL16a, JL16b, JL17, JLS16, C16, Ts17, GGS17a, GGS17b}. 

In the case of $\GL_n$, the nilpotent orbits are in one-to-one correspondence with partitions of $n$ (see \cite{CM93}). 
In the 1970s, Shalika \cite{S74} and Piatetski-Shapiro \cite{PS79} proved independently that any irreducible cuspidal automorphic representation $\pi$ has a non-zero Whittaker-Fourier coefficient, i.e. $\mathfrak{p}^m(\pi)=\{[n]\}$, corresponding to the largest nilpotent orbit.
By the work of M\oe glin and Waldspurger \cite{MW89}, the discrete spectrum of $\GL_{n}(\BA)$ consists of Speh representations $\Delta(\tau,b)$ (see \S \ref{subsection: Speh} for details), where $\tau$ runs over irreducible unitary cuspidal automorphic representations of $\GL_a(\BA)$, and $n=ab$. Ginzburg proved in \cite{G06} that $\mathfrak{p}^m(\Delta(\tau,b))=\{[a^b]\}$ with a local-global argument, and Jiang and the first-named author proved the same result in \cite{JL13} using purely global methods. 
Let $n=\sum_{i=1}^r b_i$ and consider the representation
$$\pi=\Ind_{P_{b_1,\cdots, b_r}(\BA)}^{\GL_n(\BA)} \delta_{P_{b_1,\cdots, b_r}}^{\underline{s}}$$
with $\underline{s}=(s_1,\cdots, s_r)\in \BC^r$ and $\mathrm{Re}(s_i-s_{i+1}) \gg 0$, which can be
realized as a space of degenerated Eisenstein series. Then it was conjectured by Ginzburg (see  \cite[Conjecture 5.1]{G06}) and proved recently by Cai in \cite{C16} that $$\mathfrak{p}^m(\pi)=\{[b_1 b_2\cdots b_r]^t\}
=\{[1^{b_1}] + [1^{b_2}] + \cdots + [1^{b_r}]\}\,.$$
Recall that for any partition $[q_1 q_2\cdots q_n]$ of $n$, one defines its transpose $[q_1 q_2\cdots q_n]^t$ to be $[q_1^t\cdots q_n^t]$, where $q_i^t=\sharp\{j\ |\ q_j\geq i\}$  (see \cite[\S 6.3]{CM93}).
The definition of the sum of partitions is referred to \cite[Lemma 7.2.5]{CM93}

The purpose of this paper is to generalize the results above and study the top Fourier coefficients of automorphic representations of $\GL_n(\BA)$ which are induced from Speh representations
$$\Pi_{\underline{s}} = \Ind^{\GL_n(\BA)}_{P(\BA)}
\Delta(\tau_1,b_1) \lvert \cdot \rvert^{s_1} 
\otimes \cdots \otimes \Delta(\tau_r,b_r) \lvert \cdot \rvert^{s_r},$$
where $P=MN$ is a parabolic subgroup of $\GL_n$ with Levi subgroup $M$ isomorphic to $\GL_{a_1b_1} \times \cdots \times \GL_{a_rb_r}$, $\tau_i$ is an irreducible unitary cuspidal automorphic representation of $\GL_{a_i}(\BA)$,
$n=\sum_{i=1}^r a_ib_i$, and $\underline{s}=(s_1,\ldots, s_r)\in \BC^r$.  We make the following technical assumption which will be used in the proof of Proposition \ref{nonvanishing}:

\begin{assu}\label{assumption}
Assume that $\underline{s}=(s_1,\ldots, s_r)\in \BC^r$ satisfies
\begin{enumerate}
\item $\mathrm{Re}(s_i-s_j) \in (-\infty, -1] \cup [0,1) \cup (1, \infty)$, if $b_i$ and $b_j$ have the same parity;
\item $\mathrm{Re}(s_i-s_j) \in (-\infty, -\frac{3}{2}] \cup [-\frac{1}{2},\frac{1}{2}) \cup (\frac{1}{2}, \infty)$, if $b_i$ is odd and $b_j$ is even. 
\end{enumerate}
\end{assu} 

The following is our main result. 
\begin{thm}\label{main thm}
Assume that $\underline{s}$ satisfies Assumption \ref{assumption}. Then for the automorphic representation $\Pi_{\underline{s}}$ as above, we have
$$
\mathfrak{p}^m(\Pi_{\underline{s}})
= \{[b_1^{a_1} \cdots b_r^{a_r}]^t\}
= \{[a_1^{b_1}]+\cdots +[a_r^{b_r}]\} \,.
$$
\end{thm}

Theorem \ref{main thm} proves a conjecture of Jiang (\cite[Conjecture 4.1]{J14}, see Corollary \ref{Jiang conjecture}) on top Fourier coefficients of isobaric automorphic representations of general linear groups of form:
\begin{align*}
\Pi\ &= \ \Delta(\tau_1, b_1) \boxplus \Delta(\tau_2, b_2) \boxplus \cdots \boxplus \Delta(\tau_r, b_r)\\
&\cong  \ \Ind^{\GL_n(\BA)}_{P(\BA)}
\Delta(\tau_1,b_1)
\otimes \cdots \otimes \Delta(\tau_r,b_r)\,.
\end{align*}
In this case, we have $\underline{s}=(0,\ldots, 0)$, and Assumption \ref{assumption} holds automatically.
Note that from the Arthur classification of the discrete spectrum of classical groups (see \cite{A13, M15, KMSW14, Xu14}), endoscopic lifting images of automorphic representations of classical groups occurring in the discrete spectrum form a special class of such isobaric automorphic representations. 
We also note that a related conjecture on top Fourier coefficients of Eisenstein series and their residues on $\GL_n(\BA)$ is stated in \cite[Conjecture 5.6]{G06}. 

Our proof makes use of some recent results on Fourier coefficients of automorphic forms which pack up some systematical arguments in this topic, and hence can be done in a shorter length. 
We use a result of Gomez-Gourevitch-Sahi in \cite{GGS17a} to show that $\Pi_{\underline{s}}$ has a nonzero {\it generalized Whittaker-Fourier coefficient} attached to the partition 
$[a_1^{b_1}]+\cdots +[a_r^{b_r}]$ in \S \ref{section:nonvanishing} (see Proposition \ref{nonvanishing}, and see \S \ref{def of FC's} for the definition of such Fourier coefficients). 
On the other hand, to show $[a_1^{b_1}]+\cdots +[a_r^{b_r}]$ is exactly the top orbit for $\Pi_{\underline{s}}$, one also needs to show that $\Pi_{\underline{s}}$ has no non-zero generalized Whittaker-Fourier coefficients attached to any partition bigger than or not related to $[a_1^{b_1}]+\cdots +[a_r^{b_r}]$. For this, we use a local criterion (see Proposition \ref{Degenerate Whittaker and Semi-Whittaker}) which is due to the works of M\oe glin-Waldspurger (\cite{MW87}) and Varma (\cite{V14})). 
This local criterion reduces the proof of vanishing Fourier coefficients to a simpler local vanishing statement, which is proved in \S \ref{section: vanishing} 
(see Proposition \ref{theorem for vanishing}) 
using Bernstein's localization principle (see \cite{BZ76}) and a combinatorial result of Cai (\cite{C16}).
We note that one important feature of (the constituents of) the representation $\Pi_{\underline{s}}$ we are considering is that its global top orbit equals to its local top orbit at almost all places, so that this approach works.

We remark that when we were finishing up this paper, we noticed that the same result for isobaric automorphic representations $\Pi$ as above is proved at the same time by Tsiokos in \cite{Ts17}, independently, using a different method. 

Finally, it is worthwhile to mention that, towards understanding Fourier coefficients of automorphic representations in the discrete spectrum of classical groups, in \cite[\S 4.4]{J14}, Jiang made a conjecture on the connection between Fourier coefficients of automorphic representations in an Arthur packet and the structure of the corresponding Arthur parameter
(see \cite{JL16a} for the progress on the cases of symplectic groups). 
The result of this paper will facilitate the study of Fourier coefficients of automorphic representations in the discrete spectrum of classical groups, since the endoscopic lifting image of each Arthur packet is an isobaric automorphic representation of a general linear group.

\subsection*{Acknowledgements}
We would like to thank Professor Dihua Jiang for his interest in this work and for the valuable suggestions and constant encouragement. We also thank Yuanqing Cai for helpful communication on the result in his paper \cite{C16}.

\section{Generalized and degenerate Whittaker-Fourier coefficients attached to nilpotent orbits}\label{def of FC's}

In this section, we recall the generalized and degenerate Whittaker-Fourier coefficients attached to nilpotent orbits, and also some related basic definitions mentioned in \S \ref{section: intro}, following the formulation in \cite{GGS17a}. 
Then we introduce a local criterion due to \cite{MW87, V14} on determining the top generalized Whittaker models in the case of $\GL_n$.

\subsection{The generalized and degenerate Whittaker-Fourier coefficients}\label{subsection: Gen and Deg FC}

Let $\RG$ be a reductive group defined over a number field $F$.
Fix a nontrivial additive character $\psi: F\bs \BA \rightarrow \BC^{\times}$.
Let $\mathfrak{g}$ be the Lie algebra of $\RG(F)$ and $u$ be a nilpotent element in $\mathfrak{g}$.
The element $u$ defines a function on $\mathfrak{g}$:
\[
\psi_u: \mathfrak{g} \rightarrow \BC^{\times}
\]
by $\psi_u(x) = \psi(\kappa(u,x))$, where $\kappa$ is the Killing form on $\mathfrak{g}$.

Given any semi-simple element $s \in \mathfrak{g}$, under the adjoint action, $\mathfrak{g}$ is decomposed into a direct sum of eigen-spaces $\mathfrak{g}^s_i$ corresponding to eigenvalues $i$.
The element
$s$ is called {\it rational semi-simple} if all its eigenvalues are in $\BQ$.
Given a nilpotent element $u$ and a simi-simple element $s$ in $\mathfrak{g}$, the pair $(s,u)$ is called a {\it Whittaker pair} if $s$ is a rational semi-simple element, and $u \in \mathfrak{g}^s_{-2}$. The element $s$ in a Whittaker pair $(s, u)$ is called a {\it neutral element} for $u$ if there is a nilpotent element $v \in \mathfrak{g}$ such that $(v,s,u)$ is an $\mathfrak{sl}_2$-triple. A Whittaker pair $(s, u)$ with $s$ being a neutral element is called a {\it neutral pair}. 

Given any Whittaker pair $(s,u)$, define an anti-symmetric form $\omega_u$ on $\mathfrak{g}\times \mathfrak{g}$ by 
$$\omega_u(X,Y):=\kappa(u,[X,Y])\,.$$
For any rational number $r \in \BQ$, let $\mathfrak{g}^s_{\geq r} = \oplus_{r' \geq r} \mathfrak{g}^s_{r'}$. 
 Let $\mathfrak{u}_s= \mathfrak{g}^s_{\geq 1}$ and let $\mathfrak{n}_{s,u}$ be the radical of $\omega_u |_{\mathfrak{u}_s}$. Then $[\mathfrak{u}_s, \mathfrak{u}_s] \subset \mathfrak{g}^s_{\geq 2} \subset \mathfrak{n}_{s,u}$. 
For any $X \in \mathfrak{g}$, let $\mathfrak{g}_X$ be the centralizer of $X$ in $\mathfrak{g}$.  
By \cite[Lemma 3.2.6]{GGS17a}, one has $\mathfrak{n}_{s,u} = \mathfrak{g}^s_{\geq 2} + \mathfrak{g}^s_1 \cap \mathfrak{g}_u$.
Note that if the Whittaker pair $(s,u)$ comes from an $\mathfrak{sl}_2$-triple $(v,s,u)$, then $\mathfrak{n}_{s,u}=\mathfrak{g}^s_{\geq 2}$. Let 
$N_{s,u}=\exp(\mathfrak{n}_{s,u})$ be the corresponding unipotent subgroups of $\RG$, we define a character of $N_{s,u}$ by $$\psi_u(n)=\psi(\kappa(u,\log(n))).$$ 

Let $\pi$ be an irreducible automorphic representation of $\RG(\BA)$. For any $\phi \in \pi$, the {\it degenerate Whittaker-Fourier coefficient} of $\phi$ attached to 
a Whittaker pair $(s,u)$ is defined to be
\begin{equation}\label{dwfc}
\CF_{s,u}(\phi)(g):=\int_{[N_{s,u}]} \phi(ng) \psi_u^{-1}(n) \, \mathrm{d}n\,.
\end{equation}
If $(s,u)$ is a neutral pair, then $\CF_{s,u}(\phi)$ is also called a {\it generalized Whittaker-Fourier coefficient} of $\phi$.
Let 
$$\CF_{s,u}(\pi)=\{\CF_{s,u}(\phi)\ |\ \phi \in \pi\}\,.$$
The {\it wave-front set} $\mathfrak{n}(\pi)$ of $\pi$ is defined to be the set of nilpotent orbits $\CO$ such that $\CF_{s,u}(\pi)$ is nonzero for some neutral pair $(s,u)$ with $u \in \CO$. Note that if $\CF_{s,u}(\pi)$ is nonzero for some neutral pair $(s,u)$ with $u \in \CO$, then it is nonzero for any such neutral pair $(s,u)$, since the non-vanishing property of such Fourier coefficients does not depend on the choices of representatives of $\CO$.
Moreover, we let $\mathfrak{n}^m(\pi)$ be the set of maximal elements in $\mathfrak{n}(\pi)$ under the natural ordering of nilpotent orbits (i.e., the dominance ordering).

We recall \cite[Theorem C]{GGS17a} as follows.

\begin{prop}[Theorem C, \cite{GGS17a}]\label{ggsglobal1}
Let $\pi$ be an automorphic representation of $\RG(\BA)$.
Given a neutral pair $(s,u)$ and a Whittaker pair $(s',u)$, if $\CF_{s',u}(\pi)$ is nonzero, then $\CF_{s,u}(\pi)$ is nonzero.
\end{prop}

In the rest of the paper, we consider the case of $\mathrm{G}=\GL_n$. In this case, nilpotent orbits are in one-to-one correspondence with partitions of $n$. Given a partition $\mu$ of $n$, by a Fourier coefficient of an automorphic form $\phi$ attached to $\mu$, we mean a generalized Whittaker-Fourier coefficient of $\phi$ attached to the corresponding nilpotent orbit. Given an automorphic representation $\pi$ of $\GL_n(\BA)$, let 
$\mathfrak{p}(\pi)$ and $\mathfrak{p}^m(\pi)$ be the set of partitions parameterizing nilpotent orbits in 
$\mathfrak{n}(\pi)$ and $\mathfrak{n}^m(\pi)$, respectively. Unless otherwise mentioned, by default, we use the natural ordering (i.e., the dominance ordering) for partitions. 



For latter use, we introduce a particular degenerated Whittaker-Fourier coefficients in the case of $\GL_n$. 
Let $\lambda=[p_1 p_2 \cdots p_m]$ be a partition of $n$. 
Let 
$$u_\lambda = \frac{1}{2n} \left(\sum_{i=1}^{m} \sum_{j=1}^{p_i-1} (e_{j+1} - e_j) (1)\right)$$
be a representative of the nilpotent orbit corresponding to $\lambda$, and let $s_n$ be the semi-simple element
$$
\diag\left(\frac{n-1}{2}, \frac{n-3}{2}, \ldots, \frac{1-n}{2}\right).
$$
Then $(s_n,u_\lambda)$ is a Whittaker pair. Here the multiplication of $\frac{1}{2n}$ in $u_{\lambda}$ is due to the difference between the Killing form and the trace form for general linear Lie algebras. 
For an automorphic form $\phi$ on $\GL_n(\BA)$, we will consider the degenerate Whittaker-Fourier coefficient
$$\CF_{s_n,u_\lambda}(\phi)(g):=\int_{[N_{s_n,u_\lambda}]} \phi(ng) \psi_{u_\lambda}^{-1}(n) \, \mathrm{d}n\,.$$
We note that this is the {\it $\lambda$-semi-Whittaker coefficient of $\phi$} defined in \cite{C16}. 

\medskip

\subsection{A criterion on determining local top orbits}\label{subsection: semi-Whittaker}

The generalized and degenerate Whittaker-Fourier coefficients also have their local analogues, which are certain local models. 
Let $k$ be a local field. For an irreducible smooth admissible representation $\pi$ of $\GL_n(k)$, we say that $\pi$ has a non-zero {\it degenerate Whittaker model} attached to a Whittaker pair $(s,u)$ if 
\begin{equation}\label{Whittaker model}
	\Hom_{N_{s,u}(k)}(\pi, \psi_{u})\neq 0\, .
\end{equation}
Here $N_{s,u}$ and $\psi_u$ have the same definitions as in the global setting in \S \ref{subsection: Gen and Deg FC}, and we use the same convention for admissible representations as in \cite[\S 1.1]{GGS17a}.
Moreover, if $(s,u)$ is a neutral pair, then we say that $\pi$ has a non-zero {\it generalized Whittaker model} attached to $(s,u)$ in case that (\ref{Whittaker model}) holds. 
We also have the analogous definitions for $\mathfrak{n}(\pi)$, $\mathfrak{n^m}(\pi)$, $\mathfrak{p}(\pi)$, and $\mathfrak{p^m}(\pi)$, respectively.




We have the following criterion for $\mathfrak{p^m}(\pi)$:

\begin{prop}\label{Degenerate Whittaker and Semi-Whittaker}
Let $\mu=[p_1 p_2 \cdots p_m]$ be a partition of $n$. Let $\pi$ be an irreducible admissible representation of $\GL_n(k)$,
then the following are equivalent:
\begin{enumerate}
	\item $\mathfrak{p^m}(\pi)=\{\mu\}$; 
	\item the representation $\pi$ has a non-zero degenerate Whittaker model attached to the Whittaker pair $(s_n, u_{\mu})$, and has no non-zero degenerate Whittaker model attached to the Whittaker pair $(s_n, u_{\lambda})$ for any partition
    $\lambda$ of $n$ which is bigger than or not related to $\mu$. 
\end{enumerate}
\end{prop}
\begin{proof}
The criterion is a special case of the general results in \cite{MW87} and \cite{V14}.
\end{proof}

\begin{rmk} \label{rmk0}
The more recent works of Gomez, Gourevitch and Sahi $($\cite{GGS17a, GGS17b}$)$ generalize the works in \cite{MW87} and \cite{V14}, and hence also give the above local criterion.
In \cite{C16}, Cai also suggested a global criterion 
$($see \cite[Proposition 5.3]{C16}$)$. However, we found that there is a gap in the argument for \cite[Lemma 5.7]{C16}, where the non-trivial orbit in the expansion of the inner integral can not always give the Fourier coefficient for the claimed larger partition.
As pointed out to us by Cai, this global criterion can be deduced from the global results \cite[Theorem C]{GGS17a} and \cite[Theorem 8.0.3]{GGS17b}.
\end{rmk}

\section{Certain automorphic representations of $\GL_n$}\label{section: isobaric sum}

\subsection{Structure of discrete spectrum for $\GL_n$}\label{subsection: Speh}

It was a conjecture of Jacquet (\cite{J84}) and then a theorem of M\oe glin and Waldspurger (\cite{MW89})
that an irreducible automorphic representation $\pi$ of $\GL_n(\BA)$ occurring in the discrete spectrum of the space of all square-integrable
automorphic forms on $\GL_n(\BA)$ is parameterized by a pair $(\tau,b)$ with $\tau$ being an irreducible unitary cuspidal automorphic representation
of $\GL_a(\BA)$ such that $n=ab$. In particular, we have $b=1$ if $\pi$ is cuspidal.

More precisely, for $n=ab$ with $b>1$, we take the standard parabolic subgroup $P_{a^b}=M_{a^b}N_{a^b}$ of $\GL_{ab}$, with the Levi part
$M_{a^b}$ isomorphic to
$\GL_a^{\times b}=\GL_a\times\cdots\times\GL_a$ ($b$ copies).
Following the theory of Langlands (see \cite{L76} and \cite{MW95}), there is an Eisenstein series
$E(\phi_{\tau^{\otimes b}}, \underline{s},g)$ attached to the cuspidal datum $(P_{a^b},\tau^{\otimes b})$ of $\GL_{ab}(\BA)$, where
$\underline{s}=(s_1,\cdots,s_b)\in\BC^b$. This Eisenstein series converges absolutely for the real part of $\underline{s}$ belonging to
a certain cone and has a meromorphic continuation to the whole complex space $\BC^b$. Moreover, it has an iterated residue at
$$
\underline{s}_0=\Lambda_b:=(\frac{1-b}{2},\frac{3-b}{2},\cdots,\frac{b-1}{2})\,,
$$
which is given by
\begin{equation}\label{res}
E_{-1}(\phi_{\tau^{\otimes b}}, g) =
\lim_{\underline{s} \rightarrow \Lambda_b} \prod_{i=1}^{b-1} (s_{i+1} - s_{i}-1) E(\phi_{\tau^{\otimes b}}, \underline{s}, g)\,.
\end{equation}
It is square-integrable, and hence belongs to the discrete spectrum of the space of all square-integrable
automorphic forms of
$\GL_{ab}(\BA)$. Denote by $\Delta(\tau,b)$ the automorphic representation generated by all the residues $E_{-1}(\phi_{\tau^{\otimes b}}, g)$.
M\oe glin and Waldspurger (see \cite{MW89}) proved that $\Delta(\tau,b)$ is irreducible, and
any irreducible non-cuspidal automorphic representation occurring in the discrete spectrum of the general linear group $\GL_n(\BA)$ is of this form
for some $a\geq 1$ and $b>1$ such that $n=ab$, and has multiplicity one. 
Moreover, The representation $\Delta(\tau,b)$ can be regarded as the {\it unique} irreducible subrepresentation of the induced representation
$$
\pi_{\tau,b}:=\Ind_{P_{a^b}(\BA)}^{\GL_n(\BA)} \tau \lvert \cdot \rvert^{\frac{1-b}{2}} \otimes \tau \lvert \cdot \rvert^{\frac{3-b}{2}} \otimes \cdots \otimes \tau \lvert \cdot \rvert^{\frac{b-1}{2}}\,.
$$
Here the notation $\lvert \cdot \rvert$ stands for $\lvert \det(\cdot) \rvert$ for short.

Let $\ell_b=\lceil \frac{b}{2} \rceil$ and $k_b=\lfloor \frac{b}{2}\rfloor$. 
Define $\iota_{\tau,b}$ to be the evaluation map
\begin{align*}
s_b^{(\ell_b)} & \mapsto \frac{1-b}{2}\,,\\
s_b^{(\ell_b-1)} & \mapsto \frac{3-b}{2}\,,\\
& \cdots\\
h_b^{(k_b-1)} & \mapsto \frac{b-3}{2}\,,\\
h_b^{(k_b)} & \mapsto \frac{b-1}{2}\,,
\end{align*}
on a set of parameters 
$\{s_b^{(\ell_b)}, \ldots, s_b^{(1)}, h_b^{(1)}, \ldots, h_b^{(k_b)}\}$ with $b$ entries.
Let 
\begin{equation}\label{rep with parameters}
\pi_{\tau,b}' := \Ind_{P_{a^b}(\BA)}^{\GL_n(\BA)} \tau \lvert \cdot \rvert^{s_b^{(\ell_b)}} \otimes \cdots \otimes \tau \lvert \cdot \rvert^{s_b^{(1)}} \otimes \tau \lvert \cdot \rvert^{h_b^{(1)}} \otimes \cdots \otimes \tau \lvert \cdot \rvert^{h_b^{(k_b)}}\,.\end{equation}
Then the map $\iota_{\tau,b}$ naturally induces a map from 
$\pi_{\tau,b}'$ to $\pi_{\tau,b}$, which we still denote by $\iota_{\tau,b}$.

\subsection{Certain automorphic representations of $\GL_n$}\label{subsection:isobaric sum}

Write $n=\sum_{i=1}^r a_i b_i$, where $a_i$ and $b_i$ are both positive integers. 
For $1 \leq i\leq r$, let $\tau_i$ be an irreducible unitary cuspidal automorphic representation of $\GL_{a_i}(\BA)$,
and $\Delta(\tau_i,b_i)$ be the corresponding representation in the discrete spectrum of $\GL_{a_i b_i}(\BA)$.
Let $P=M N$ be a parabolic subgroup of $\GL_n$ with Levi subgroup $M$ isomorphic to $\GL_{a_1b_1} \times \cdots \times \GL_{a_r b_r}$. 
In this paper, we mainly consider induced representations
$$\Pi_{\underline{s}} = \Ind^{\GL_n(\BA)}_{P(\BA)}
\Delta(\tau_1,b_1) \lvert \cdot \rvert^{s_1} 
\otimes \cdots \otimes \Delta(\tau_r,b_r) \lvert \cdot \rvert^{s_r},$$
where $\underline{s}=(s_1,\cdots, s_r)\in \BC^r$ satisfies Assumption \ref{assumption}.
By Langlands' theory of Eisenstein series (see \cite{L76, L79a}), as automorphic representations, each constituent of $\Pi_{\underline{s}}$ is realized via meromorphic continuation of certain Eisenstein series or their residues. 

Denote
$$
\Pi=\Delta(\tau_1, b_1) \boxplus \Delta(\tau_2, b_2) \boxplus \cdots \boxplus \Delta(\tau_r, b_r)
$$
to be the induced representation 
$$
\Ind_{P(\BA)}^{\GL_n(\BA)} \Delta(\tau_1, b_1) \otimes \cdots \otimes \Delta(\tau_r, b_r)\,.
$$
Then $\Pi$ is an irreducible unitary automorphic representation of $\GL_n(\BA)$. 
In sense of \cite[Section 2]{L79b} (see also \cite[Section 1.3]{A13}), the representation $\Pi$ is {\it isobaric}.

The main purpose of this paper is to study the top generalized Whittaker-Fourier coefficients of $\Pi_{\underline{s}}$ (see Theorem \ref{main thm}). In particular, we will 
prove the following conjecture proposed by Jiang in \cite{J14} on top generalized Whittaker-Fourier coefficients of $\Pi$. This is a direct corollary of Theorem \ref{main thm} in \S 1.  

\begin{cor}[Conjecture 4.1, \cite{J14}]\label{Jiang conjecture}
For isobaric automorphic representations 
$\Pi$ as above, we have
$$
\mathfrak{p}^m(\Pi)
= \{[a_1^{b_1}]+\cdots +[a_r^{b_r}]\} \,.
$$
\end{cor}

\section{Proof of Theorem \ref{main thm}: non-vanishing for the top orbit}\label{section:nonvanishing}

In this section, we show that the representation
$$
\Pi_{\underline{s}} = \Ind^{\GL_n(\BA)}_{P(\BA)}
\Delta(\tau_1,b_1) \lvert \cdot \rvert^{s_1} 
\otimes \cdots \otimes \Delta(\tau_r,b_r) \lvert \cdot \rvert^{s_r}
$$
has a nonzero generalized Whittaker-Fourier coefficient attached to the partition 
$$[a_1^{b_1}]+\cdots +[a_r^{b_r}]\,.$$

The main idea is to show that $\Pi_{\underline{s}}$ is a subquotient of a representation induced from certain parabolic subgroup and generic data. Before carrying out the argument, we first explain the steps using an explicit example. 
For convenience, we will denote the induced representation
$\Pi_{\underline{s}}$ as 
$$\Delta(\tau_1,b_1) \lvert \cdot \rvert^{s_1} 
\times \cdots \times \Delta(\tau_r,b_r) \lvert \cdot \rvert^{s_r}.$$ 

\begin{exmp}
	We consider representation
$$
\Pi_{\underline{s}}=\Delta(\tau_1, 3) \lvert \cdot \rvert^{s_1} \times \Delta(\tau_2, 4)\lvert \cdot \rvert^{s_2} \times \Delta(\tau_3, 5)\lvert \cdot \rvert^{s_3}\,,
$$
where $\tau_i$ is a unitary cuspidal representation of $\GL_{a_i}(\BA)$, $1 \leq i \leq 3$. 
Note that 
$\Delta(\tau_1, 3)$ is the unique irreducible subrepresentation of 
$$\pi_{\tau_1,3}=\tau_1 \lvert \cdot \rvert^{-1} \times \tau_1 \lvert \cdot \rvert^{0}
\times \tau_1 \lvert \cdot \rvert^{1}\,;$$
$\Delta(\tau_2, 4)$ is the unique irreducible subrepresentation of 
$$\pi_{\tau_2,4}=\tau_2 \lvert \cdot \rvert^{-\frac{3}{2}} \times \tau_2 \lvert \cdot \rvert^{-\frac{1}{2}}
\times \tau_2 \lvert \cdot \rvert^{\frac{1}{2}} \times \tau_2 \lvert \cdot \rvert^{\frac{3}{2}} \,;$$
and $\Delta(\tau_3, 5)$ is the unique irreducible subrepresentation of 
$$\pi_{\tau_3,5}=\tau_3 \lvert \cdot \rvert^{-2} \times \tau_3 \lvert \cdot \rvert^{-1}\times \tau_3 \lvert \cdot \rvert^{0}
\times \tau_3 \lvert \cdot \rvert^{1} \times \tau_3 \lvert \cdot \rvert^{2} \,.$$
Then $\Pi_{\underline{s}}$ is a subquotient of 
$$\pi_{\tau_1,3}\lvert \cdot \rvert^{s_1} \times \pi_{\tau_2,4}\lvert \cdot \rvert^{s_2} \times \pi_{\tau_3,5}\lvert \cdot \rvert^{s_3}\,.$$

We put all the inducing data in a table as follows:
\begin{center}
\begin{tabular}{ |c|c|c|c|c| } 
 \hline
  & $\tau_1 \lvert \cdot \rvert^{-1+s_1}$ & $\tau_1 \lvert \cdot \rvert^{0+s_1}$ & $\tau_1 \lvert \cdot \rvert^{1+s_1}$ &  \\ 
 \hline
 & $\tau_2 \lvert \cdot \rvert^{-\frac{3}{2}+s_2}$ & $\tau_2 \lvert \cdot \rvert^{-\frac{1}{2}+s_2}$ & $\tau_2 \lvert \cdot \rvert^{\frac{1}{2}+s_2}$ & $\tau_2 \lvert \cdot \rvert^{\frac{3}{2}+s_2}$ \\ 
 \hline
 $\tau_3 \lvert \cdot \rvert^{-2+s_3}$ & $\tau_3 \lvert \cdot \rvert^{-1+s_3}$ & $\tau_3 \lvert \cdot \rvert^{0+s_3}$ & $\tau_3 \lvert \cdot \rvert^{1+s_3}$ & $\tau_3 \lvert \cdot \rvert^{2+s_3}$ \\ 
 \hline
\end{tabular}
\end{center}
For $\Delta(\tau_1,3)$ and $\Delta(\tau_3,5)$, we put the inducing components $
\tau_i\lvert\cdot\rvert^{0+s_i}$ into the same column, and for $\Delta(\tau_2,4)$, we also put $\tau_2\lvert\cdot \rvert^{-\frac{1}{2}+s_2}$ into the same column as above. 
Then we put the other inducing components, with order unchanged, into the corresponding rows. 
The placement of the inducing data of $\Delta(\tau_2,4)$ is not unique, for example, we can also put $\tau_2\lvert\cdot \rvert^{\frac{1}{2}+s_2}$ into the center column of the above table. 
The point is, if we rearrange all the inducing data by columns of the above table, i.e., let 
\begin{align*}
\eta_1 & = \tau_1 \lvert \cdot \rvert^{-1+s_1} \times 
\tau_2 \lvert \cdot \rvert^{-\frac{3}{2}+s_2} \times
\tau_3 \lvert \cdot \rvert^{-1+s_3}\,,\\
\eta_2 & = \tau_1 \lvert \cdot \rvert^{0+s_1} \times 
\tau_2 \lvert \cdot \rvert^{-\frac{1}{2}+s_2} \times
\tau_3 \lvert \cdot \rvert^{0+s_3}\,,\\
\eta_3 & = \tau_1 \lvert \cdot \rvert^{1+s_1} \times 
\tau_2 \lvert \cdot \rvert^{\frac{1}{2}+s_2} \times
\tau_3 \lvert \cdot \rvert^{1+s_3}\,,\\
\eta_4 & =  
\tau_2 \lvert \cdot \rvert^{\frac{3}{2}+s_2} \times
\tau_3 \lvert \cdot \rvert^{2+s_3} \,,\\
\eta_5 & =  
\tau_3 \lvert \cdot \rvert^{-2+s_3}\,,
\end{align*}
the $\eta_i$'s are irreducible generic representations of certain general linear groups under Assumption \ref{assumption}, and $\Pi_{\underline{s}}$ is a subquotient of 
$$\eta_1 \times \eta_2 \times \eta_3 \times \eta_4 \times \eta_5\,.$$
See the proof of Proposition \ref{nonvanishing} for more details.
Note that $\Pi_{\underline{s}}$ has a nonzero constant term with respect to the parabolic subgroup whose Levi subgroup is 
$$\GL_{a_1+a_2+a_3}^3(\BA) \times \GL_{a_2+a_3}(\BA) \times \GL_{a_3}(\BA)\,.$$
Hence,  
$\Pi_{\underline{s}}$ has a nonzero degenerate Whittaker-Fourier coefficient attached to the partition $[(a_1+a_2+a_3)^3(a_2+a_3)a_3]$ (see \S \ref{subsection: Gen and Deg FC} for the definition). 
By Proposition \ref{ggsglobal1}, $\Pi_{\underline{s}}$ also has a nonzero generalized Whittaker-Fourier coefficient attached to the partition $[(a_1+a_2+a_3)^3(a_2+a_3)a_3]$, which is exactly $[a_1^3]+[a_2^4]+[a_3^5]$.
\end{exmp}

Now we carry out the general argument and prove the following proposition. 

\begin{prop}\label{nonvanishing}
Assume that $s_i$'s satisfy Assumption \ref{assumption}. Then $\Pi_{\underline{s}}$ has a nonzero generalized Whittaker-Fourier coefficient attached to the partition $\mu=[a_1]^{b_1} + \cdots + [a_r]^{b_r}$.
\end{prop}

\begin{proof}
Without loss of generality, we may assume that 
$b_1\geq b_2\geq \cdots \geq b_r$.
Then we can write $\mu = [t_1 t_2 \cdots t_{b_1}]$ with $t_1 \geq t_2 \geq \cdots \geq t_{b_1}$.
Recall from \S \ref{subsection:isobaric sum} that given an integer $b$, we have defined $\ell_b=\lceil \frac{b}{2} \rceil$ and $k_b=\lfloor \frac{b}{2}\rfloor$. 
For $1 \leq i \leq r$, we form parameters 
$$\left(s_{b_i}^{(\ell_{b_1})}, \ldots, s_{b_i}^{(1)}, h_{b_i}^{(1)}, \ldots, h_{b_i}^{(k_{b_1})}\right)\, $$
with $b_1$ entries by adding zeros from the front if $\ell_{b_i} < \ell_{b_1}$, and adding zeros from the end if $k_{b_i} < k_{b_1}$ (note that $b_1\geq b_j$ for $2\leq j\leq r$).
In other words, one has $s_{b_i}^{(j)}=0$ if $j>\ell_{b_i}$, and  $h_{b_i}^{(j)}=0$ if $j>k_{b_i}$.
Then, for $1 \leq j \leq k_{b_1}$, we construct representations 
$$
\sigma_j^1 = \tau_1 \lvert \cdot \rvert^{h_{b_1}^{(j)}+s_1} \times \cdots \times \tau_r \lvert \cdot \rvert^{h_{b_r}^{(j)}+s_r}\,,
$$
where we omit the $\tau_i \lvert \cdot \rvert^{h_{b_i}^{(j)}+s_i}$-term if $j>k_{b_i}$ ($1 \leq i \leq r$). Similarly, for $1 \leq q \leq \ell_{b_1}$, we construct representations 
$$
\rho_q^1 = \tau_1 \lvert \cdot \rvert^{s_{b_1}^{(q)}+s_1} \times \cdots \times \tau_r \lvert \cdot \rvert^{s_{b_r}^{(q)}+s_r}\,, 
$$
where we omit the $\tau_i \lvert \cdot \rvert^{s_{b_i}^{(q)}+s_i}$-term if $q>\ell_{b_i}$ ($1 \leq i \leq r$). 

For $1 \leq j \leq k_{b_1}$, we assume that $\sigma_j^1$ is a representation of $\GL_{n_j}(\BA)$,  and for $1 \leq q \leq \ell_{b_1}$, we assume that $\rho_q^1$ is a representation of $\GL_{m_q}(\BA)$. 
Note that the $n_j$'s and $m_q$'s are among the integers 
$$\left\{\sum_{i=1}^r \delta_i a_i\ |\ \delta_i=0\ \text{or}\ 1\right\}.$$
Then we have
$$[n_1\, n_2\, \ldots n_{k_{b_1}}\, m_1\, m_2 \,\ldots \, m_{\ell_{b_1}}] = [t_1\, t_2\, \ldots \, t_{b_1}\,] \,.$$
We rename the representations $\{\sigma_j^1, 1 \leq j \leq k_{b_1}, \rho_q^1, 1 \leq q \leq \ell_{b_1}\}$ as 
$\{\varepsilon_1, \ldots, \varepsilon_{b_1}\}$,  such that 
$\varepsilon_i$ is a representation of $\GL_{t_i}(\BA)$ ($1 \leq i \leq b_1$). 

Now we can rewrite the induced representation (recall that $\pi_{\tau_i,b_i}'$ is defined in (\ref{rep with parameters}))
$$
\Ind_{P(\BA)}^{\GL_n(\BA)} \pi_{\tau_1,b_1}' \lvert \cdot \rvert^{s_1} \otimes \cdots \otimes \pi_{\tau_r,b_r}'\lvert \cdot \rvert^{s_r}
$$
equivalently as 
$$
\Ind_{Q(\BA)}^{\GL_n(\BA)} \varepsilon_1 \otimes \cdots \otimes \varepsilon_{b_1}\,,
$$
where $Q=LV$ is the parabolic subgroup of $\GL_n$ with Levi subgroup $L\subset Q$ isomorphic to $\GL_{t_1} \times \cdots \times \GL_{t_{b_1}}$. 

Recall that we have defined maps $\iota_{\tau_i, b_i}$ in \S \ref{subsection: Speh}.
Then the set of maps $$\{\iota_{\tau_1,b_1}, \ldots, \iota_{\tau_r,b_r}\}$$ 
naturally induces a map $I_1$
on the representation
$$\Ind_{P(\BA)}^{\GL_n(\BA)}  \pi_{\tau_1,b_1}'\lvert \cdot \rvert^{s_1} \otimes \cdots \otimes \pi_{\tau_r,b_r}'\lvert \cdot \rvert^{s_r}\,,$$
and hence induces a map $I_2$ on the representation
$$
\Ind_{Q(\BA)}^{\GL_n(\BA)} \varepsilon_1 \otimes \cdots \otimes \varepsilon_{b_1}\,.
$$
By construction, one has ($\pi_{\tau_i,b_i}$ is also  defined in \S \ref{subsection: Speh})
$$
\begin{aligned}
&\quad \ I_1(\Ind_{P(\BA)}^{\GL_n(\BA)}  \pi_{\tau_1,b_1}' \lvert \cdot \rvert^{s_1}\otimes \cdots \otimes \pi_{\tau_r,b_r}'\lvert \cdot \rvert^{s_r})\\
&=\Ind_{P(\BA)}^{\GL_n(\BA)}  \pi_{\tau_1,b_1}\lvert \cdot \rvert^{s_1} \otimes \cdots \otimes \pi_{\tau_r,b_r}\lvert \cdot \rvert^{s_r}\,,
\end{aligned}$$
and  
$$
I_2 (\Ind_{Q(\BA)}^{\GL_n(\BA)} \varepsilon_1 \otimes \cdots \otimes \varepsilon_{b_1})=\Ind_{Q(\BA)}^{\GL_n(\BA)} \eta_1 \otimes \cdots \otimes \eta_{b_1}\,,$$
where $\eta_i$ is a representation of $\GL_{t_i}(\BA)$ ($1 \leq i \leq b_1$), coming from evaluating $s_{b_i}^{(q)}$'s and $h_{b_i}^{(j)}$'s in
$\varepsilon_i$.

Note that each $\eta_i$ is equivalent to a representation of the form
$$\tau_{\kappa_1} \lvert \cdot \rvert^{e_1+s_{\kappa_1}} \times \cdots \times \tau_{\kappa_\alpha} \lvert \cdot \rvert^{e_\alpha+s_{\kappa_\alpha}} \times \tau_{\gamma_1} \lvert \cdot \rvert^{f_1+s_{\gamma_1}} \times \cdots \times \tau_{\gamma_\beta} \lvert \cdot \rvert^{f_\beta+s_{\gamma_\beta}}\,,$$
where $1\leq \alpha+\beta\leq r$, 
$b_{\kappa_i}$'s are odd and $b_{\gamma_j}$'s are even, $e_i$'s and $f_j$'s are half-integers such that $e_1=\cdots=e_\alpha$, $f_1=\cdots=f_\beta$,
and $e_i-f_j =1/2$. 
We claim that each $\eta_i$ is an irreducible generic representation of $\GL_{t_i}(\BA)$.
Indeed, consider the Eisenstein series corresponding to an induced representation 
$$\rho_1 \lvert \cdot \rvert^{\nu_1} \times \cdots \times \rho_k \lvert \cdot \rvert^{\nu_k}\quad  (\nu_i \in \BC)\, ,$$
with $\rho_i$'s being irreducible unitary cuspidal automorphic representations. 
The calculation of constant term (see, for example, \cite[Chapter 6]{Sh10}) implies that the poles of the Eisenstein series are given by 
the ratio of Rankin-Selberg $L$-functions $$\prod_{1 \leq i<j \leq k}\frac{L(\nu_i-\nu_j, \rho_i\times \widetilde{\rho}_{j})}{L(1+\nu_i-\nu_j, \rho_{i}\times \widetilde{\rho}_{j})}\,.$$
By \cite[Appendice, Proposition and Corollaire]{MW89},
$L(\nu_i-\nu_j, \rho_{i}\times \widetilde{\rho}_{j})$ has only simple poles at $\nu_i-\nu_j=0, 1$,  and by \cite{JS76, JS81, Sh80, Sh81} (see also \cite[Theorem 4.3]{Cog07}), $L(\nu_i-\nu_j, \rho_{i}\times \widetilde{\rho}_{j})$ is non-vanishing for $\mathrm{Re}(\nu_i-\nu_j) \geq 1$ or $\mathrm{Re}(\nu_i-\nu_j) \leq 0$. Since $s_i$'s satisfy  Assumption \ref{assumption}, 
it is clear that the Eisenstein series corresponding to $\eta_i$ has no pole at the point 
$$(e_1+s_{\kappa_1},\cdots, e_\alpha+s_{\kappa_\alpha}, f_1+s_{\gamma_1},\cdots, f_\beta+s_{\gamma_\beta})\,.$$
Hence each $\eta_i$ is irreducible generic.


Recall that 
$\Delta(\tau_i, b_i)$ is the unique irreducible subrepresentation of 
$\pi_{\tau_i,b_i}$, then
$$
 \Pi_{\underline{s}} = \Ind_{P(\BA)}^{\GL_n(\BA)} \Delta(\tau_1, b_1)\lvert \cdot \rvert^{s_1} \otimes \cdots \otimes \Delta(\tau_r, b_r)\lvert \cdot \rvert^{s_r}
$$
is a subquotient of 
$$\Ind_{P(\BA)}^{\GL_n(\BA)}  \pi_{\tau_1,b_1}\lvert \cdot \rvert^{s_1} \otimes \cdots \otimes \pi_{\tau_r,b_r}\lvert \cdot \rvert^{s_r}\,,$$
hence, is also a subquotient of 
$$\Ind_{Q(\BA)}^{\GL_n(\BA)} \eta_1 \otimes \cdots \otimes \eta_{b_1}\,.$$
Let $P_{a_1^{b_1}, \cdots, a_r^{b_r}}$ be the parabolic subgroup of $\GL_n$ with Levi subgroup isomorphic to $\GL_{a_1}^{\times b_1} \times \cdots \times \GL_{a_r}^{\times b_r}$. 
Then one sees that $\Pi_{\underline{s}}$ has a nonzero constant term with respect to $P_{a_1^{b_1}, \cdots, a_r^{b_r}}$ by unfolding of Eisenstein series and the cuspidal support of $\Pi_{\underline{s}}$. 
Since the parabolic subgroup $Q$ contains the parabolic subgroup $P_{a_1^{b_1}, \cdots, a_r^{b_r}}$,  
$\Pi_{\underline{s}}$ also has a nonzero constant term with respect to $Q$. 

Let $$\displaystyle{u_{\mu} = \sum_{i=1}^{b_1} \sum_{j=1}^{t_i-1} (e_{j+1} - e_j) (1)}$$ 
be a representative of the nilpotent orbit $\CO$ corresponding to the partition $\mu=[t_1 t_2 \ldots t_{b_1}]$, and let $s$ be the semi-simple element 
$$
\diag\left(\frac{t_1-1}{2}, \ldots, \frac{1-t_1}{2}, \frac{t_2-1}{2}, \ldots, \frac{1-t_2}{2}, \ldots, \frac{t_{b_1}-1}{2}, \ldots, \frac{1-t_{b_1}}{2}\right).
$$
It is easy to see that $s$ is a neutral element for $u$, and hence $(s,u)$ is a neutral pair. 
Recall that we have defined another semi-simple element 
$$
s_n=\diag\left(\frac{n-1}{2}, \frac{n-3}{2}, \ldots, \frac{1-n}{2}\right)
$$
in \S \ref{subsection: Gen and Deg FC}, and $(s_n,u_{\mu})$ is also a Whittaker pair. 

Take $0 \neq f \in \Pi_{\underline{s}}$, and consider the degenerate Fourier coefficient
$\CF_{s_n,u_{\mu}}(f)$
attached to the Whittaker pair $(s_n,u_{\mu})$. It is easy to see that $\CF_{s_n,u_{\mu}}(f)$ is the constant term integral along the parabolic subgroup $Q$ composed with the Whittaker-Fourier coefficient along the Levi subgroup $L$. Note that $\Pi_{\underline{s}}$ is a constituent of the induced representation 
$$\Ind_{Q(\BA)}^{\GL_n(\BA)} \eta_1 \otimes \cdots \otimes \eta_{b_1}\,,$$ 
and $\eta_i$ is an irreducible generic representation of $\GL_{t_i}(\BA)$ for
$1 \leq i \leq b_1$. We claim that $\CF_{s_n,u_{\mu}}(f) \neq 0$. 
Indeed, the constant term of $f$ along $Q$ gives us a nonzero vector in the irreducible generic representation $\eta_1 \otimes \cdots \otimes \eta_{b_1}$ of $L(\BA)$, whose Whittaker-Fourier coefficient is hence nonzero. 
Then by Proposition \ref{ggsglobal1} we also have $\CF_{s,u_{\mu}}(f) \neq 0$, that is, $\Pi_{\underline{s}}$ has a nonzero generalized Whittaker-Fourier coefficient attached to the partition $\mu$.

This completes the proof of the proposition. 
\end{proof}

\begin{rmk} 
We give some remarks on the proof of Proposition \ref{nonvanishing}:
\begin{enumerate}
\item The same argument of the proof shows that any constituent $\pi$ of the induced representation 
$\Ind_{Q(\BA)}^{\GL_n(\BA)} \eta_1 \otimes \cdots \otimes \eta_{b_1}$ has a nonzero generalized Whittaker-Fourier coefficient attached to the partition $[a_1^{b_1}]+\cdots +[a_r^{b_r}]$. However, this partition may not be a maximal partition providing nonzero generalized Whittaker-Fourier coefficients for $\pi$. The main result of this paper shows that, for the constituents of $\Pi_{\underline{s}}$, this partition is indeed the maximal partition providing nonzero generalized Whittaker-Fourier coefficients.
\item By Generalized Riemann Hypothesis (GRH), given irreducible unitary cuspidal automorphic representations $\rho_1, \rho_2$, the zeros of the $L$-function $L(s,\rho_1 \times \rho_2)$ only lie on the line $\Re(s)=\frac{1}{2}$. Granting this, our Assumption \ref{assumption} can be simplified to be:
    \begin{itemize}
      \item[(i)] $\Re(s_i-s_j)\neq -\frac{1}{2}, 1$, if $b_i$ and $b_j$ have the same parity; 
      \item[(ii)] $\Re(s_i-s_j)\neq \frac{1}{2}, -1$, if $b_i$ is odd and $b_j$ is even.
    \end{itemize}
\item We apply the general result in \cite{GGS17a} (Proposition \ref{ggsglobal1}) to show that $\Pi_{\underline{s}}$ has a nonzero generalized Whittaker-Fourier coefficient attached to the partition $\mu=[a_1^{b_1}]+\cdots +[a_r^{b_r}]$, from the result that $\CF_{s_n,u_{\mu}}(f)\neq 0$ for some $f\in \Pi_{\underline{s}}$. 
Similar arguments have been used in \cite{JLX18} in the study of certain twisted automorphic descent constructions and a reciprocal problem of the global Gan-Gross-Prasad conjecture, and is expected to to applied in some more general situation.  
\end{enumerate}
\end{rmk}

\section{Proof of Theorem \ref{main thm}: vanishing for bigger and not related orbits}\label{section: vanishing}

In this section, we show that
$$
\Pi_{\underline{s}} = \Ind^{\GL_n(\BA)}_{P(\BA)}
\Delta(\tau_1,b_1) \lvert \cdot \rvert^{s_1} 
\otimes \cdots \otimes \Delta(\tau_r,b_r) \lvert \cdot \rvert^{s_r}
$$
has no nonzero generalized Whittaker-Fourier coefficients attached to any partition either bigger than or not related to the partition $$\mu=[a_1^{b_1}]+\cdots +[a_r^{b_r}]\,.$$
Combining with Proposition \ref{nonvanishing},
this completes the proof of Theorem \ref{main thm}.

Assume that $b_1\geq b_2\geq \cdots \geq b_r$ as before, then
$$\mu= [(a_1+\cdots+a_r)^{b_r}(a_1+\cdots a_{r-1})^{b_{r-1}-b_r}\cdots (a_1+a_2)^{b_2-b_3} a_1^{b_1-b_2}]\,.$$
Note that $\mu^t = [b_1^{a_1} \cdots b_r^{a_r}]$. 
For any partition $\nu=[d_1 d_2 \cdots d_l]$ of $n$, we let $P_\nu$ be the standard parabolic subgroup of $\GL_N$ whose Levi subgroup is $M_\nu\simeq \GL_{d_1}\times \cdots \times \GL_{d_l}$,
and denote the corresponding unipotent subgroup by $N_\nu$.

The main result in this section is the following.

\begin{prop}\label{theorem for vanishing}
The representation $\Pi_{\underline{s}}$ has no nonzero generalized Whittaker-Fourier coefficients attached to any partition either bigger than or not related to the partition $\mu=[a_1^{b_1}]+\cdots +[a_r^{b_r}]$.
\end{prop}

\begin{proof}
For simplicity, write $\mu=[t_1t_2\cdots t_{b_1}]$, $t_1 \geq \cdots \geq t_{b_1}$. 
Let $\Pi$ be any constituent of $\Pi_{\underline{s}}$. 
By Proposition \ref{nonvanishing} and its proof, $\Pi$ has a nonzero degenerate Whittaker-Fourier coefficient attached to the Whittaker pair $(s_n, u_{\mu})$, hence, for any finite place $v$, $\Pi_{v}$ has a nonzero degenerate Whittaker model attached to this same Whittaker pair. Therefore, by Proposition \ref{Degenerate Whittaker and Semi-Whittaker}, we just need to show that, at some finite place $v$, $\Pi_{v}$ has no nonzero degenerate Whittaker model attached to the Whittaker pair $(s_n, u_{\lambda})$, 
for any partition $\lambda=[p_1 p_2 \cdots p_m]$ ($p_1 \geq \cdots \geq p_m$) of $n$ which is bigger than or not related to $\mu$.
Note that for any such partition $\lambda$, there exists $1 \leq i \leq m$ such that $p_1+\cdots +p_i > t_1+\cdots +t_i$.

By \cite[Lemma 1]{L79a}, constituents of $\Pi_{\underline{s}}$ are pairwise nearly equivalent. 
	We consider the local unramified components of any constituent $\Pi$.
	Let $v$ be a finite place such that $\Pi_{v}$ is unramified.
	For $1\leq i \leq r$, write  $\tau_{i,v}=\chi_1^{(i)}\times\cdots \times \chi_{a_i}^{(i)}$ with 
	$\chi_j^{(i)}$'s being unramified characters of $F_v^\times$, then $\Pi_{v}$ is the unique irreducible unramified constituent of the following induced representation
    $$\Ind_{P_{\mu^t}(F_v)}^{\GL_n(F_v)}\sigma_{1,v}\lvert \cdot \rvert^{s_1} \otimes\cdots \otimes \sigma_{r,v}\lvert \cdot \rvert^{s_r} \, ,$$ 
    where 
	$$\sigma_{i,v}=\chi_1^{(i)}(\det\!_{\GL_{b_i}})\times\cdots \times \chi_{a_i}^{(i)}(\det\!_{\GL_{b_i}})$$
	is a representation of $\GL_{a_i b_i}(F_v)$. 
    
    Write $\varrho_v=\sigma_{1,v}\lvert \cdot \rvert^{s_1} \otimes\cdots \otimes \sigma_{r,v}\lvert \cdot \rvert^{s_r} ,$
	we claim that	
	\begin{equation}
		\Hom_{U(F_v)}(\Ind_{P_{\mu^t}(F_v)}^{\GL_n(F_v)} \varrho_v, \psi_{u_\lambda,v}) = 0
	\end{equation}
	provided that there exists $1 \leq i \leq m$ such that $p_1+\cdots +p_i > t_1+\cdots +t_i$. This implies that $\Pi_{v}$ has no nonzero degenerate Whittaker model attached to the Whittaker pair $(s_n, u_{\lambda})$, for any partition $\lambda=[p_1 p_2 \cdots p_m]$ ($p_1 \geq \cdots \geq p_m$) of $n$ which is bigger than or not related to $\mu$. 
    Recall that $\psi_{u_\lambda,v}$ is defined in \S \ref{subsection: semi-Whittaker}.

	We use Bernstein's localization principle (see \cite[\S 6]{BZ76}) to study the $\Hom$-space
	\begin{equation}\label{key hom-space}
    \Hom_{U(F_v)}(\Ind_{P_{\mu^t}(F_v)}^{\GL_n(F_v)} \varrho_v, \psi_{u_\lambda,v})\,.
    \end{equation}
    By Frobenius reciprocity, this $\Hom$-space is equivalent to the space 
    \begin{equation}\label{key hom-space II}
    \Hom_{\GL_n(F_v)}(\Ind_{P_{\mu^t}(F_v)}^{\GL_n(F_v)} \varrho_v, \Ind_{U(F_v)}^{\GL_n(F_v)}\psi_{u_\lambda,v})\,.
    \end{equation}
    
    Let $$\CD(P_{\mu^t}(F_v)\bs \GL_n(F_v))$$ be the space of distributions on $P_{\mu^t}(F_v)\bs \GL_n(F_v)$, i.e., the complex linear functionals on $C_c^\infty(P_{\mu^t}(F_v)\bs \GL_n(F_v))$. 
    Let $$T\in \CD(P_{\mu^t}(F_v)\bs \GL_n(F_v))$$ be the distribution associated to the induced representation $$\Ind_{P_{\mu^t}(F_v)}^{\GL_n(F_v)} \varrho_v$$ in the sense of \cite[\S 6]{BZ76}.    
	Consider the right action of $U(F_v)$ on $P_{\mu^t}(F_v)\bs \GL_n(F_v)$ and the restriction of the distribution $T$ to the double coset $P_{\mu^t}(F_v) w U(F_v)$ with $w\in P_{\mu^t}(F_v)\bs \GL_n(F_v)/ U(F_v)$. 
	Such restriction is associated to the following compact induced representation 
	$$\ind_{U(F_v)\cap (w^{-1}P_{\mu^t}(F_v)w)}^{U(F_v)} \varrho_v^w\, ,$$ 
	where $\varrho_v^w(g)=\varrho_v(w g w^{-1})$ for $g\in U(F_v)\cap w^{-1}P_{\mu^t}(F_v)w$.
	By Frobenius reciprocity,
	\begin{equation}\label{localized Hom-spaces}
		\begin{split}
		&\quad\ \Hom_{U(F_v)}(\ind_{U(F_v)\cap (w^{-1}P_{\mu^t}(F_v)w)}^{U(F_v)} \varrho_v^w, \psi_{u_\lambda,v})\\
		&\simeq \Hom_{U(F_v)\cap (w^{-1}P_{\mu^t}(F_v)w)}(\varrho_v^w,\psi_{u_\lambda,v})\,. \end{split}
	\end{equation}
	Note that by construction, we have $\varrho_v^w|_{U(F_v)\cap (w^{-1}P_{\mu^t}(F_v)w)}\equiv 1$. 
	Moreover, by a root-theoretic result \cite[Theorem 1.3]{C16}, if there exists $1 \leq i \leq m$ such that $p_1+\cdots +p_i > t_1+\cdots +t_i$, then for any representative $w\in P_{\mu^t}(F_v) \bs \GL_n(F_v)/ U(F_v)$, there exists $u\in U(F_v)$ such that $\psi_{u_\lambda,v}(u)\neq 1$ and $wuw^{-1}\in P_{\mu^t}(F_v)$. 
	Therefore, the right hand side of (\ref{localized Hom-spaces}), and hence its left hand side, is identically zero for all $w\in P_{\mu^t}(F_v) \bs \GL_n(F_v)/ U(F_v)$.
		
	On the other hand, by \cite[Theorem 6.15]{BZ76}, the right action of $U(F_v)$ on the quotient $P_{\mu^t}(F_v)\bs \GL_n(F_v)$ is constructive. 
	Then by Bernstein's localization principle (see \cite[Theorem 6.9]{BZ76}), we have
	$$\Hom_{U(F_v)}(\Ind_{P_{\mu^t}(F_v)}^{\GL_n(F_v)} \varrho_v, \psi_{u_\lambda,v}) = 0\, .$$
	This proves the claim above and completes the proof of the proposition.	 
\end{proof}


\begin{thebibliography}{XXXX}

\addtocontents{Bibliography}

\bibitem[A13]{A13}
J. Arthur, {\it The Endoscopic Classification of Representations: Orthogonal and Symplectic Groups},
	American Mathematical Society Colloquium Publications, 2013.

\bibitem[BZ76]{BZ76}
I. N. Bernstein and A. V. Zelevinski, {\it Representations of $\GL(n,F)$, where $F$ is a non-Archimedean local field}, Russ. Math. Surveys \textbf{31}, No. 3, 1--68 (1976)

\bibitem[C16]{C16}
Y. Cai, {\it Fourier coefficients for degenerate Eisenstein series and the descending decomposition}, to appear in manuscripta mathematica, (2017).



\bibitem[Cog07]{Cog07}
J. Cogdell, {\it $L$-functions and converse theorems for $\GL_n$}. Automorphic forms and applications, 97--177, IAS/Park City Math. Ser., \textbf{12}, Amer. Math. Soc., Providence, RI, 2007. 

\bibitem[CM93]{CM93}
D. Collingwood and W. McGovern,
{\it Nilpotent orbits in semisimple Lie algebras.}
Van Nostrand Reinhold Mathematics Series. Van Nostrand Reinhold Co., New York, 1993. xiv+186 pp.

\bibitem[G06]{G06}
D. Ginzburg, {\it Certain conjectures relating unipotent orbits to automorphic representations}, Israel Journal of Mathematics, \textbf{151} (2006), pp. 323--355.



\bibitem[GRS03]{GRS03}
D. Ginzburg, S. Rallis and D. Soudry,
{\it On Fourier coefficients of automorphic forms of symplectic groups.}
Manuscripta Math. \textbf{111} (2003), no. 1, 1--16.


\bibitem[GRS11]{GRS11}
D. Ginzburg, S. Rallis and D. Soudry,
{\it The descent map from automorphic representations of {${\rm GL}(n)$} to classical groups.} World Scientific, Singapore, 2011. v+339 pp.



\bibitem[GGS17a]{GGS17a}
R. Gomez, D. Gourevitch and S. Sahi,
{\it Generalized and degenerate Whittaker models}.
Compositio Math. \textbf{153} (2017) 223–-256.

\bibitem[GGS17b]{GGS17b}
R. Gomez, D. Gourevitch and S. Sahi,
{\it Whittaker supports for representations of reductive groups}.
Preprint. 2017. arXiv: 1610.00284.




%


\bibitem[J84]{J84}
H. Jacquet, 
\textit{On the residual spectrum of ${\rm GL}(n)$}. Lie group representations, II (College Park, Md., 1982/1983), 185--208, Lecture Notes in Math., \textbf{1041}, Springer, Berlin, 1984.

\bibitem[JS76]{JS76}
H. Jacquet and J. Shalika,
{\it A non-vanishing theorem for zeta functions of $\GL_n$}. 
Invent. Math. \textbf{38} (1976/77), no. 1, 1--16.

\bibitem[JS81]{JS81}
H. Jacquet and J. Shalika,
{\it On Euler products and the classification of automorphic representations}. I. Amer. J. Math. \textbf{103} (1981), no. 3, 499--558.

\bibitem[J14]{J14}
D. Jiang,
{\it Automorphic Integral transforms for classical groups I: endoscopy correspondences}.
Automorphic Forms: L-functions and related geometry: assessing the legacy of I.I. Piatetski-Shapiro.
Comtemp. Math. \textbf{614}, 2014, AMS.

\bibitem[JL13]{JL13}
D. Jiang and B. Liu,
{\it On Fourier coefficients of automorphic forms of ${\rm GL}(n)$}.
Int. Math. Res. Not. 2013 (17): 4029--4071.

\bibitem[JL15]{JL15}
D. Jiang and B. Liu,
{\it On special unipotent orbits and Fourier coefficients for automorphic forms on symplectic groups}.
J. of Number Theory, \textbf{146} (2015), 343--389.

\bibitem[JL16a]{JL16a}
D. Jiang and B. Liu,
{\it Arthur parameters and Fourier coefficients for automorphic forms on symplectic groups.}
Ann. Inst. Fourier (Grenoble), \textbf{66} (2016), no. 2, 477--519.


\bibitem[JL16b]{JL16b}
D. Jiang and B. Liu,
{\it Fourier coefficients for automorphic forms on quasisplit classical groups.}
Advances in the Theory of Automorphic Forms and Their L-functions.
Contemporary Mathematics, Volume \textbf{664}, 2016, 187--208, AMS. 

\bibitem[JL17]{JL17}
D. Jiang and B. Liu,
{\it Fourier coefficients and cuspidal spectrum for symplectic groups.}
To appear in Geometric aspects of the trace formula, Simons Symposia. 2017. 


\bibitem[JLS16]{JLS16}
D. Jiang, B. Liu, and G. Savin,
{\it Raising nilpotent orbits in wave-front sets.}
Representation Theory \textbf{20} (2016), 419--450.

\bibitem[JLX18]{JLX18}
D. Jiang, B. Liu and B. Xu,
{\it A reciprocal problem of the Gan-Gross-Prasad conjecture for even orthogonal groups}, submitted (2018).





\bibitem[KMSW14]{KMSW14}
T. Kaletha, A. Minguez, S. W. Shin and P.-J. White,
{\it Endoscopic Classification of Representations: Inner Forms of Unitary Groups.}
arXiv:1409.3731 (2014).

\bibitem[L76]{L76}
R. Langlands,
{\it On the functional equations satisfied by Eisenstein series}. Springer Lecture Notes in Math. \textbf{544}. 1976.

\bibitem[L79a]{L79a}
R. Langlands,
\textit{On the notion of an automorphic representation. A supplement to the preceding paper}. Automorphic forms, representations and $L$-functions
(Proc. Sympos. Pure Math., Oregon State Univ., Corvallis, Ore., 1977), Part 1, pp. 203-207, Proc. Sympos. Pure Math., XXXIII,
Amer. Math. Soc., Providence, R.I., 1979.

\bibitem[L79b]{L79b}
R. Langlands,
{\it Automorphic representations, Shimura varieties, and motives.} Ein Marchen. Automorphic forms, representations and L-functions (Proc. Sympos. Pure Math., Oregon State Univ., Corvallis, Ore., 1977), Part 2, pp. 205--246, Proc. Sympos. Pure Math., XXXIII, Amer. Math. Soc., Providence, R.I., 1979.

%







\bibitem[MW87]{MW87}
C. M\oe glin and J.-L. Waldspurger,
{\it Mod\`eles de Whittaker d\'eg\'en\'er\'es pour des groupes $p$-adiques}. (French)
Math. Z. \textbf{196} (1987), no. 3, 427--452.

\bibitem[MW89]{MW89}
C. M\oe glin and J.-L. Waldspurger,
{\it  Le spectre residuel de {${\rm GL}(n)$.}} Ann. Sci. \'Ecole Norm. Sup. (4) \textbf{22} (1989), no. 4, 605--674.

\bibitem[MW95]{MW95}
C. M\oe glin and J.-L. Waldspurger,
{\it Spectral decomposition and Eisenstein series.} Cambridge Tracts in Mathematics, \textbf{113}. Cambridge University Press, Cambridge, 1995.

\bibitem[M15]{M15}
C. P. Mok, 
{\it Endoscopic classification of representations of quasi-split unitary groups.}
Mem. of AMS, \textbf{235} (2015), No. 1108.

%


\bibitem[PS79]{PS79}
I. Piatetski-Shapiro,
\textit{Multiplicity one theorems}. Automorphic forms, representations and $L$-functions
(Proc. Sympos. Pure Math., Oregon State Univ., Corvallis, Ore., 1977), Part 1, pp. 209--212, Proc. Sympos. Pure Math., XXXIII,
Amer. Math. Soc., Providence, R.I., 1979.

\bibitem[Sh80]{Sh80}
F. Shahidi,
{\it On nonvanishing of L-functions.}
Bull. Amer. Math. Soc. (N.S.) \textbf{2} (1980), no. 3, 462--464.

\bibitem[Sh81]{Sh81}
F. Shahidi,
{\it On certain L-functions.}
Amer. J. Math. \textbf{103} (1981), no. 2, 297--355.

\bibitem[Sh10]{Sh10}
F. Shahidi,
{\it Eisenstein series and automorphic L-functions,}
volume 58 of American Mathematical Society Colloquium Publications. American Mathematical Society, Providence, RI, 2010. ISBN 978-0-8218- 4989-7.

 



\bibitem[S74]{S74}
J. Shalika,
\textit{The multiplicity one theorem for ${\rm GL}_{n}$}. Ann. of Math. (2) \textbf{100} (1974), 171--193.


\bibitem[Ts17]{Ts17}
E. Tsiokos, 
{On Fourier Coefficients of $\mathrm{GL}(n)$-Automorphic Functions over Number Fields.}
Preprint 2017. arXiv: 1711.11545. 

\bibitem[V14]{V14}
S. Varma,
{\it On a result of Moeglin and Waldspurger in residual characteristic 2}, Math. Z. \textbf{277}, no. 3--4, 1027--1048 (2014).


\bibitem[W01]{W01}
J.-L. Waldspurger,
{\it Int\'egrales orbitales nilpotentes et endoscopie pour les groupes classiques non ramifi\'es}. Ast\'erisque \textbf{269}, 2001.

\bibitem[Xu14]{Xu14}
B. Xu, 
{\it Endoscopic classification of representations of $\mathrm{GSp}(2n)$ and $\mathrm{GSO}(2n)$.}
PhD Thesis, 2014, University of Toronto.


\end{thebibliography}
\end{document}